\documentclass[12pt,a4paper]{amsart}
\usepackage{amssymb,enumerate,mathtools,cite}
\usepackage{graphicx,epstopdf,wasysym}
\usepackage{caption,subcaption,float}
\usepackage{hyperref} %to link to the citations
\numberwithin{equation}{section}
\newtheorem{theorem}{Theorem}[section]
\newtheorem{lemma}[theorem]{Lemma}

\theoremstyle{remark}
\newtheorem{remark}{Remark}[section]
\newtheorem{definition}{Definition}[section]
\newcommand{\D}{\mathbb{D}}
\newcommand{\C}{\mathbb{C}}

\allowdisplaybreaks

\DeclareMathOperator{\RE}{Re}

 \def \a{\alpha}

 \title[Certain subclasses of Analytic Functions]{Certain subclasses of Analytic Functions with Fixed Second Coefficient}

  \author [S. Kumar]{Sushil Kumar}
 \address {Bharati Vidyapeeth's college of Engineering, Delhi-110063, India}
 \email{sushilkumar16n@gmail.com}

 \author[S. Anand]{Swati Anand}

 \address{Department of Mathematics, University of Delhi,
 Delhi--110 007, India}
 \email{swati\_anand01@yahoo.com}
 \author [N.K. Jain]{Naveen Kumar Jain}
 \address {Department of Mathematics, Aryabhatta College, Delhi-110021,India}
 \email{naveenjain05@gmail.com}

\begin{document}
\subjclass[2010]{30C80, 30C45}

\keywords{Analytic functions with fixed second coefficient;  starlike functions;  parabolic starlike functions; radius estimates.}
\maketitle
\begin{abstract}
This paper aims to pursue some classes of normalized analytic functions $f$ with fixed second coefficient defined on open unit disk, such that ${(1+z)^2f(z)}/{z}$  and ${(1+z)f(z)}/{z}$ are functions having positive real part.
The radius of strongly starlikeness, the radius of lemniscate starlikeness, the radius of parabolic starlikeness and other starlikeness estimates are calculated for such functions.
As well relevant connections of computed radii estimates with the existing one are also shown.
\end{abstract}

\section{Introduction}
Let $\mathcal{A}$ be the class of analytic functions $f$ defined in  $\D= \{z\in \mathbb{C}: \vert z\vert <1\}$ satisfying  $f(0) = f'(0)-1=0$ and $\mathcal{S}$ be the sub class of $\mathcal{A}$ contains univalent. Let $\mathcal{M}$ be a set of functions and ${P}$ be a  property. The supremum $R$ of all the radii  so that each function $f \in \mathcal{M}$ has the property $P$ in the disk $\mathbb{D}_r$, $0\leq r\leq R$
is known as radius of property for the set $\mathcal{M}$, denoted by $R_{P}(\mathcal{M})$. Thus, it is noted  that $R_{S^*}(\mathcal{S})$ is $\tanh\frac{\pi}{4}\approx 0.65579$.
Refer ~\cite{lee1,juneja,kumar21,cho19} for more literature on radius problems.
For a fixed real number $b$ in $[0,1]$, the class $\mathcal{A}_b$ consists  of  analytic functions
\[f(z) = z+bz^2+ \cdots \quad\quad (z \in \D).\]
Since the bound on second coefficient plays a momentous role in the univalent function theory,
the functions having fixed second coefficients  attracted the interest of many authors/readers.
In the beginning, Gronwall~\cite{Gronwall} studied the univalent functions having fixed second coefficient.  In ~\cite{ANV}, authors discussed second order differential subordination for such functions.
In 2013, Lee \emph{et al.} accorded some utilizations
of  subordination for such functions~\cite{lee}. Recently,  the best possible estimates on  initial coefficients of Ma-Minda type univalent functions were determined by Kumar \emph{et al.}  ~\cite{sushil17}.
Growth and distortion  estimates, subordination and radius problems for functions $f \in \mathcal{A}_b$ have been discussed in~\cite{ahuja87, mnr2}.
Using the  Hadamard product `$\ast$' and subordination `$\prec$', Padmanabhan and Parvatham \cite{padman} studied the   functions $f \in \mathcal{A}$ fulfilling subordination criteria  $z(c_\a \ast f)'/(c_\a \ast f) \prec \phi$ where $c_\a(z) = z/(1-z)^\a$, $\a$ is a real number and $\phi$ is a convex function.
 Shanmugan~\cite{TN1} investigated the unified class $S^*_j(\phi)$ that contains  functions $f\in \mathcal{A}$ satisfying $z(f \ast j)'/(f \ast j) \prec \phi$ where $\phi$ is convex and $j$ is a specified function in $\mathcal{A}$. If $j(z) = z/(1-z)$,  the class $S^*_j(\phi)$ yields the class  $S^*(\phi)$~\cite{MM}.
For sample, if $\phi(z) = (1+Az)/(1+Bz)$ fulfilling condition  $-1 \le B <A\le1$, then $S^*(\phi)$  scale down to  $S^*[A,B]$. Recently, Anand \emph{et al.}  \cite{ravi} studied the class $S^*[A,B]$ to obtain various radius problems.
In particular, if $A= 1-2\a$ and $B=-1$ fulfilling condition $0 \le \a <1$, then  $S^*(\a)$ contains starlike functions of order $\a$.
For $0 \le \gamma <1$, the class $S^*_{\gamma}$ contains strongly stalike functions of order $\gamma$ satisfying the inequality
$\left| \mbox{arg} ({zf'(z)}/{f(z)})\right|<({\alpha \pi}/{2})$. The class $S_p:= S^*(\phi_{\mathcal{PAR}}(z))$, where $\phi_{\mathcal{PAR}}(z)=1+({2}/{\pi^2}) \left(\log (({1-\sqrt{z}})/({1+\sqrt{z}}))\right)^2$,  contains parabolic
starlike functions which are associated with parabolic region $\{w\in \C:|w-1|<\RE(w)\}$~\cite{Ronn93}. The class $ S^*_L:=S^*(\sqrt{1+z})$ contains lemniscate starlike functions allied with the lemniscate shaped region  $\{w\in \C:|w^2-1|<1\}$ ~\cite{SS1}.
The classes $S^*_e:= S^*(e^z)$ and $S^*_{RL}:= S^*\left(\sqrt{2}-(\sqrt{2}-1)\sqrt{(1-z)/(1+2(\sqrt{2}-1)z)}\right)$ which are associated with the domains  $\{w\in \C:|\log w|<1\}$ and $\{w\in \C:|w^2-\sqrt{2}w+1|<1\}$ respectively studied in~\cite{mnr,mnr1}. The class $S^*_c:= S^*((3+4z+2z^2)/3)$ was studied in \cite{sharma}.
The class $S^*_{\leftmoon}:=S^*(z+\sqrt{1+z^2})$  consists of lune starlike functions associated with the lune shaped region $\{w \in \C:2 |w|>|w^2-1|\}$~\cite{raina}.
Cho \emph{et al.} \cite{cho} introduced the class $S^*_{sin}$ related to the function $1+\sin z$.
The class $S^*_R:=S^*(\psi(z))$ is associated with the function $\psi(z) = 1+(z(k+z))/(k(k-z))$ where $k = \sqrt{2}+1$ was also made known in ~\cite{kumar}. The class $S^*_{N_e}:= S^*(1+z-(z^3/3))$
is associated with the two cusped kidney shaped region $ \Omega_{N_e}:= \{u+ ib:(u^2-2u+v^2+5/9)^3-4v^2/3<0\}$~\cite{wani}.  Recently, the class $S^*_{SG}:=S^*(2/(1+e^{-z}))$ associated with   $\{w \in \C:|\log( (2/(2-w))-1)|<1\}$ is also studied in~\cite{goel}.
 For the functions $f$ with fixed second coefficient, Ali \emph{et al.} ~\cite{ali1} computed radii of starlikeness. Further, authors~\cite{ali2} examined  radius estimates for the analytic functions $f$  satisfying some conditions on $f/g$ for some $g\in \mathcal{A}$.
In recent past, certain best possible radius estimates for certain classes of functions $f \in \mathcal{A}$ fulfilling inequality $\RE (1+z)f(z)/z >0$  and  $\RE (1+z)^2f(z)/z >0$  have been determined~\cite{AV}. For more details, see~\cite{agh18,pas89,kanaga}.

Motivated by above explained research work, we fix radius estimates for some classes considered in~\cite{AV} having fixed second coefficients to be in the classes $S^*(\a)$, $S_L^* $, $S_p $, $S_e^*$, $S_c^* $, $S_{\sin}^* $, $S_{\leftmoon}^*$, $S_R^*$, $S_{RL}^*$, $S_{\gamma}^*$, $S_{N_e}^*$ and  $S_{SG}^*$. Some connections of obtained radii estimates with the existing ones are also demonstrated.

\section{Radius estimates for the class $\mathcal{G}_b^1$}
Let the class $\mathcal{P}(\a)$ accommodate analytic functions having following series expansion
$p(z) = 1+b_1z+b_2 z^2+ \cdots$ fulfilling  $\RE(p(z)) > \a$ ($0 \le \a <1$). Note that $\mathcal{P}:= \mathcal{P}(0)$. For $p \in \mathcal{P}(\a)$\cite[P.170]{nehari}, we have
$|b_1| \le 2(1-\a).$
Let $P_b(\a)$ be a subclass of $\mathcal{P}(\a)$ consisting of functions $p$ of the form
$p(z) = 1+ 2b(1-\a)z+ \cdots, |b| \le 1$. It is noted that $\mathcal{P}_b:= \mathcal{P}_b(0)$.
Let $f$ be the  function  having  expansion
\begin{equation}\label{eqf}
f(z) = z+ a_2 z^2+ \cdots
\end{equation}
and $\RE ((1+z)^2 f(z)/z) >0$. Then $(1+z)^2 f(z)/z = 1+(2+a_2)z+ \cdots \in \mathcal{P}$ and hence $0\leq a_2 \leq 4$. Now, we consider the   class   consisting of functions   of the form $f(z)=z+4bz^2+\cdots$ ($-1\leq b\leq1$).

\begin{definition}
For $|b| \le 1$, let
\[\mathcal{G}_b^1 := \left\{f \in \mathcal{A}: f(z)=z+4bz^2+\cdots,  \RE \left(\frac{(1+z)^2}{z}f(z) \right)>0, z \in \D\right\}. \]
\end{definition}
We will use following lemma in the proof of main results:
\begin{lemma}\cite[Theorem 2, p. 213]{carty}\label{lem1}
Let $|b| \le 1$,  $0 \le \a <1$ and $p \in \mathcal{P}_b(\a)$. Then
\begin{equation*}% \label{eq1}
\left|\frac{zp'(z)}{p(z)}\right| \le \frac{2(1-\a)r}{1-r^2} \frac{|b|r^2+2r+|b|}{(1-2\a)r^2+2(1-\a)|b|r+1}\quad \mbox{for} \quad|z|=r <1.
\end{equation*}
\end{lemma}
In the following result we determind the radius constants for the class $\mathcal{G}_b^1$.

\begin{theorem}\label{ThmF4}
Let $b_1=|1+2b|\leq 1$ and  $\mathcal{G}_b^1 $ be a class of analytic functions. Then the following radius estimates are obtained:
\begin{enumerate}
\item  $R_{S^*(\a)}(\mathcal{G}_b^1)=\rho_1$, where $\rho_1$  is the smallest root in $(0,1)$ of the equation
\begin{equation}\label{eqa4}
-(1+\a )r^4+2(1-\a b_1)r^3  +2(1+2b_1)r^2+2(1+\a b_1)r+\a-1=0.
\end{equation}
The result is sharp.
\item $R_{S_L^*}(\mathcal{G}_b^1)=\rho_2$, where $\rho_2$ is the smallest root in $(0,1)$ of the equation
\begin{align}\label{eql4}
(1+\sqrt{2})r^4&+2(1+\sqrt{2}(1+\sqrt{2})b_1)r^3+2(3+2b_1)r^2 \nonumber\\
& +2(1+\sqrt{2}(\sqrt{2}-1)b_1)r+1-\sqrt{2}=0.
\end{align}
The result is sharp.
\item  $R_{S_p}(\mathcal{G}_b^1)=\rho_3$, where $\rho_3$ is the smallest root in $(0,1)$ of the equation
\begin{equation}\label{eqp4}
-3r^4+2(2-b_1)r^3+4(1+2b_1)r^2+2(2+b_1)r-1=0.
\end{equation}
The result is sharp.
\item $R_{S_e^*}(\mathcal{G}_b^1)=\rho_4$, where $\rho_4$ is the smallest root in $(0,1)$ of the equation
\begin{equation} \label{eqe4}
-(1+e)r^4+2(e-b_1)r^3 +2e(1+2b_1)r^2+2(e+b_1)r+1-e=0.
\end{equation}
The result is sharp.
\item $R_{S_c^*}(\mathcal{G}_b^1)=\rho_5$, where $\rho_5$ is the smallest root in $(0,1)$ of the equation
\begin{equation} \label{eqc4}
-4r^4 +2(3-b_1)r^3+6(1+2b_1)r^2+2(3+b_1)r-2=0.
\end{equation}
The estimate is best possible.
\item $R_{S_{\sin}^*}(\mathcal{G}_b^1)=\rho_6$, where $\rho_6$ is the smallest root in $(0,1)$ of the equation
\begin{align}\label{eqs4}
(\sin 1+2)r^4 +&2(1+(3+\sin 1)b_1)r^3+2(3+2b_1)r^2 \nonumber\\
&+2(1+(1-\sin 1)b_1)r-\sin 1=0.
\end{align}
The result is sharp.
\item $R_{S_{\leftmoon}^*}(\mathcal{G}_b^1)=\rho_7$, where $\rho_7$ is the smallest root in $(0,1)$ of the equation
\begin{equation}\label{eqm4}
 -\sqrt{2}r^4 +2(1+(1-\sqrt{2})b_1)r^3 +2(1+2b_1)r^2+2(1+(\sqrt{2}-1)b_1)r+\sqrt{2}(1-\sqrt{2})=0.
\end{equation}
\item $R_{S_R^* }(\mathcal{G}_b^1)=\rho_8$, where $\rho_8$ is the smallest root in $(0,1)$ of the equation
\begin{align}\label{eqr4}
(1-2\sqrt{2})r^4&+2(1+2(1-\sqrt{2})b_1)r^3+2(1+2 b_1)r^2\nonumber\\
&+ 2(1+2(\sqrt{2}-1)b_1)r+(2\sqrt{2}-3)=0.
\end{align}
The result is sharp.
\item $R_{S_{RL}^*}(\mathcal{G}_b^1)=\rho_9$, where $\rho_9$ is the smallest root in $(0,1)$ of the equation
\begin{align}\label{eqrl4}
&(2(1+b_1)r^3+4(1+b_1)r^2+2(1+b_1)r )^2
-(r^2+2b_1r+1)^2((1-r^2)\notag\\&\sqrt{-r^4-2  (\sqrt{2}-1 ) r^2+2  (\sqrt{2}-1 )}
+r^4+2 \left(\sqrt{2}-1\right) r^2-2 \sqrt{2}+2=0.
\end{align}
\item $R_{S_{\gamma}^*}(\mathcal{G}_b^1)=\rho_{10}$, where $\rho_{10}$ is the smallest root in $(0,1)$ of the equation
\begin{align*}%\label{eqg4}
 &-\sin(\pi \gamma/2)  r ^4 +2(1+(1-\sin(\pi \gamma/2))b_1) r ^3 \\ &+2(2(1+b_1)-\sin(\pi \gamma/2)) r ^2 +2(1+(1-\sin(\pi \gamma/2))b_1)r -\sin(\pi \gamma/2)= 0.
\end{align*}

\item $R_{S_{N_e}^*}(\mathcal{G}_b^1)=\rho_{11}$, where $\rho_{11}$ is the smallest root in $(0,1)$ of the equation
\begin{equation}\label{eqn4}
 8r^4+2(3+b_1)r^3+6(2+b_1)r^2+2(3+b_1)r-2=0.
\end{equation}
The result is sharp.
\item $R_{S_{SG}^*}(\mathcal{G}_b^1)=\rho_{12}$, where $\rho_{12}$ is the smallest root in $(0,1)$ of the equation
\begin{equation}\label{eqsg4}
 (1+3e)r^4+2(1+e+2(1+2e)b_1)r^3+2(1+e)(3+2b_1)r^2+2(1+e+2b_1)r+1-e=0.
\end{equation}
The result is sharp.
\end{enumerate}
\end{theorem}

\begin{proof}
Let  $f \in \mathcal{G}_b^1$  such that
\begin{equation} \label{eqfg3}
 \RE \left(\frac{(1+z)^2}{z}f(z) \right)>0 \,\,\mbox{for all}\,\, z \in \D.
\end{equation}
The function $h$ defined on $\D$ as
\begin{equation} \label{eqh(z)3}
h(z) = \frac{(1+z)^2}{z}f(z) =  1+2(1+2b)z+ \cdots .
\end{equation}
In view of (\ref{eqfg3}) and (\ref{eqh(z)3}), we observe that the function $h \in \mathcal{P}_{(1+2b)}$ and $f(z) = zh(z)/(1+z)^2$. Then
\begin{equation}\label{eqf(z)3}
\frac{zf'(z)}{f(z)} = \frac{zh'(z)}{h(z)}+\frac{1-z}{1+z}.
\end{equation}
The transformation $\dfrac{1-z}{1+z}$ transform  $|z| \le r$ onto
\begin{equation} \label{eqtr1}
\left|\frac{1-z}{1+z}-\frac{1+r^2}{1-r^2}\right|\le \frac{2r}{1-r^2}.
\end{equation}
Using Lemma \ref{lem1}, we have
\begin{equation} \label{eqt4}
\left|\frac{zh'(z)}{h(z)}\right| \le \frac{2r}{1-r^2} \left(\frac{b_1r^2+2r+b_1}{r^2+2b_1r+1}\right).
\end{equation}
Using (\ref{eqtr1}) and (\ref{eqt4}),  from (\ref{eqf(z)3}), we get
\begin{align}\label{eqdisk6}
\left|\frac{zf'(z)}{f(z)}- \frac{1+r^2}{1-r^2}\right| & \le \frac{2r}{1-r^2}\left(\frac{b_1r^2+2r+b_1}{r^2+2b_1r+1}+1\right)\notag \\
&= \frac{2((1+b_1)r^3+2(1+b_1)r^2+(1+b_1)r)}{(1-r^2)(r^2+2b_1r+1)}.
\end{align}
It follows by the above inequality that
\[\RE\left(\frac{zf'(z)}{f(z)}\right) \ge\frac{r^4-2r^3-2(1+2b_1)r^2-2r+1}{-r^4-2b_1r^3+2b_1r+1}  \ge 0\]
whenever
\[-r^4+2r^3+2(1+2b_1)r^2+2r-1 \le 0.\]

\begin{enumerate}
\item In view of (\ref{eqdisk6}), we have
\[\RE\left(\frac{zf'(z)}{f(z)}\right) \ge\frac{r^4-2r^3-2(1+2b_1)r^2-2r+1}{-r^4-2b_1r^3+2b_1r+1}  \ge \a\]
whenever
\begin{equation*}
-(1+\a )r^4+2(1-\a b_1)r^3  +2(1+2b_1)r^2+2(1+\a b_1)r+\a-1\le 0.
\end{equation*}
Thus, for the class $\mathcal{G}_b^1$, the $S^*(\a)$ -radius is the  root $\rho_1 \in (0,1)$ of (\ref{eqa4}).
The function $f_1$ is defined by
\begin{equation} \label{ext4}
f_1(z) = \frac{z-z^2}{(1+z)(1-2(1+2b)z+z^2)}
\end{equation}
so that
\[ \frac{(1+z)^2}{z}f_1(z) = \frac{1-w_1(z)}{1+w_1(z)}\]
where
\[w_1(z) = \dfrac{z(z-(1+2b))}{(1-(1+2b)z)}\quad \mbox{ with}\quad |2b+1| \le 1\]
is an analytic function fulfilling all the cases of Schwarz lemma in $\D$ and thus $\RE ((1+z)^2f_1(z)/z) >0$ in $\D$. Hence  $f_1\in \mathcal{G}_b^1$.
At $z=r = \rho_1$,  we have
\begin{align*}
\frac{zf_1'(z)}{f_1(z)}
%&= \frac{z^4-2z^3+(2+8b)z^2-2z+1}{(1-z^2)(1-2(1+2b)z+z^2)}\\
&=\frac{r^4-2r^3+(2+8b)r^2-2r+1}{(1-r^2)(1-2(1+2b)r+r^2)} \\
&= \a
\end{align*}
which follows the sharpness.

\item It follows from (\ref{eqdisk6}) that
\begin{align*}%\label{eqdisk7}
\left|\frac{zf'(z)}{f(z)}- 1\right| &\le \left|\frac{zf'(z)}{f(z)}- \frac{1+r^2}{1-r^2}\right|+\frac{2r^2}{1-r^2} \notag\\
 &\le  \frac{2r^4+2(1+3b_1)r^3+2(3+2b_1)r^2+2(1+b_1)r}{-r^4-2b_1r^3+2b_1r+1}.
 \end{align*}
By \cite[Lemma 2.2, p. 6559]{ali}, the function $f \in \mathcal{G}_b^1$ go to the class $S^*_L$, whenever the following inequality holds
\[\frac{2r^4+2(1+3b_1)r^3+2(3+2b_1)r^2+2(1+b_1)r}{-r^4-2b_1r^3+2b_1r+1} \le \sqrt{2}-1\]
or
\begin{align*}
(1+\sqrt{2})r^4&+2(1+\sqrt{2}(1+\sqrt{2})b_1)r^3+\\
&2(3+2b_1)r^2 +2(1+\sqrt{2}(\sqrt{2}-1)b_1)r+1-\sqrt{2} \le 0.
\end{align*}
Thus, the $S_L^*-$ radius  is the smallest root $\rho_2 \in (0,1)$ of (\ref{eql4}).
%Using triangle inequality and (\ref{eqdisk7}) we get
%\[\left|\frac{zf'(z)}{f(z)}+ 1\right| \le \sqrt{2}+1\] and hence
%\[\left|\left(\frac{zf'(z)}{f(z)}\right) ^2- 1\right|=1.\]
The function $f_2:\D \to  \C $ is defined by
\begin{equation} \label{ext5}
 f_2(z) = \frac{z(1+ 2(1+2b)z+z^2)}{(1+z)^2(1-z^2)}.
\end{equation}
so that
\[\frac{(1+z)^2}{z}f_2(z) = \frac{1+w_2(z)}{1-w_2(z)}\]
where \[w_2(z) = \dfrac{z(z+(1+2b))}{(1+(1+2b)z)}\quad \mbox{with}\,\, |1+2b| \le 1\]
is an analytic function fulfilling all the cases of Schwarz lemma in the unit disk $\D$. Thus, $\RE ((1+z)^2f_2(z)/z) >0$ in $\D$ and hence   $f_2 \in \mathcal{G}_b^1$.
The number $\rho = \rho_2$ satisfies
\[\frac{\rho^4-(2+8b)\rho^3+(2-8b)\rho^2-(2+8b)\rho+1}{(1-\rho^2)(1-2(1+2b)\rho+\rho^2)}= \sqrt{2}.\]
Thus, at $z =-\rho = -\rho_2$,  we have
\begin{align*}
\left|\left(\frac{zf_2'(z)}{f_2(z)}\right)^2-1\right| %\left|\left(\frac{z^4+(2+8b)z^3+(2-8b)z^2+(2+8b)z+1}{(1-z^2)(1+2(1+2b)z+z^2)}\right)^2-1\right|\\
&=\left|\left(\frac{\rho^4-(2+8b)\rho^3+(2-8b)\rho^2-(2+8b)\rho+1}{(1-\rho^2)(1-2(1+2b)\rho+\rho^2)}\right)^2-1\right|\\
&= 1.
\end{align*}
This shows the sharpness.

\item Using \cite[Section 3˘, p. 321]{TN},  the disk (\ref{eqdisk6}) lies inside the parabolic region $\Omega_p= \{w:\RE w >{w-1}\}$ provided that
\[\frac{2((1+b_1)r^3+2(1+b_1)r^2+(1+b_1)r)}{(1-r^2)(r^2+2b_1r+1)}
\le \frac{1+r^2}{1-r^2}-\frac{1}{2}\]
which on simplification becomes
\begin{equation*}
-3r^4+2(2-b_1)r^3+4(1+2b_1)r^2+2(2+b_1)r-1 \le 0.
\end{equation*}
Thus, the $S_p$- radius  is the smallest   root $\rho_3 \in (0,1)$ of (\ref{eqp4}).
For the function $f_1$ given in (\ref{ext4}) at $z= \rho = \rho_3$, the following calculations shows the desired sharpness
\begin{align*}
\RE \frac{zf_1'(z)}{f_1(z)} &= \frac{\rho^4-2\rho^3+(2+8b)\rho^2-2\rho+1}{(1-\rho^2)(1-2(1+2b)\rho+\rho^2)}\\
&=\frac{-2\rho^4+(4+4b)\rho^3-(2+8b)\rho^2-4b\rho}{(1-\rho^2)(1-2(1+2b)\rho+\rho^2)}\\
&=\left|\frac{2\rho^4-(4+4b)\rho^3+(2+8b)\rho^2+4b\rho}{(1-\rho^2)(1-2(1+2b)\rho+\rho^2)}\right|\\
&= \left|\frac{zf_1'(z)}{f_1(z)}-1\right|.
\end{align*}

\item By \cite[Lemma 2.2, p.\ 368]{mnr}, the function $f \in \mathcal{G}_b^1$ belongs to the class $ S_e^*$ if
\[\frac{2((1+b_1)r^3+2(1+b_1)r^2+(1+b_1)r)}{(1-r^2)(r^2+2b_1r+1)}
\le \frac{1+r^2}{1-r^2}-\frac{1}{e}\]
equivalently
\begin{equation*}
-(1+e)r^4+2(e-b_1)r^3 +2e(1+2b_1)r^2+2(e+b_1)r+1-e \le 0.
\end{equation*}
Thus, the $S_e^*$- radius   is the smallest   root $\rho_4 \in (0,1)$ of (\ref{eqe4}).
For $z=\rho=\rho_4$ and the function $f_1$, the following calculation prove the sharpness
\begin{align*}
\left|\log \frac{zf_1'(z)}{f_1(z)}\right|  &= \left|\log \left(\frac{\rho^4-2\rho^3-2(3+4b)\rho^2-2\rho+1}{(1-\rho^2)(1+2(1+2b)\rho+\rho^2)}\right)\right|=1.
\end{align*}

\item By making use of \cite[Lemma 2.5, p.926]{sharma},  the function $f \in \mathcal{G}_b^1$ belongs to the class $ S_c^*$ if
\[\frac{2((1+b_1)r^3+2(1+b_1)r^2+(1+b_1)r)}{(1-r^2)(r^2+2b_1r+1)}
\le \frac{1+r^2}{1-r^2}-\frac{1}{3}\]
or
\begin{equation*}
-4r^4 +2(3-b_1)r^3+6(1+2b_1)r^2+2(3+b_1)r-2\le 0.
\end{equation*}
Thus, the $S_c^*$- radius  is the smallest  root $\rho_5 \in (0,1)$ of (\ref{eqc4}).
At $z=\rho=\rho_5$, we have
\begin{align*}
\left|\ \frac{zf_1'(z)}{f_1(z)}\right|  &= \left| \frac{\rho^4-2\rho^3+(2+8b)\rho^2-2\rho+1}{(1-\rho^2)(1-2(1+2b)\rho+\rho^2)}\right|\\
&=\frac{1}{3}\,\, \text{which lies in the boundary of}\,\,  \Omega_c
\end{align*}
 where $\Omega_C $ is  cardioid shaped region $\{x+iy:81 x^4-324 x^3+162 x^2 y^2+270 x^2-324 x y^2-84 x+81 y^4-54 y^2+9=0\}$.

\item Using \cite[Lemma 3.3, p.219]{cho}, the function $f \in \mathcal{G}_b^1$ belongs to the class $S_{\sin}^*$ if
\[\frac{2((1+b_1)r^3+2(1+b_1)r^2+(1+b_1)r)}{(1-r^2)(r^2+2b_1r+1)}\le \sin 1- \frac{2r^2}{1-r^2}\]
 equivalently
\begin{equation*}
(2+\sin 1)r^4 +2(1+(3+\sin 1)b_1)r^3+2(3+2b_1)r^2 +2(1+(1-\sin 1)b_1)r-\sin 1\le0.
\end{equation*}
Thus, the $S_{\sin}^* $- radius  is the smallest  root $\rho_6 \in (0,1)$ of (\ref{eqs4}).
For the function $f_2$ defined in (\ref{ext5}) and $z=-\rho=-\rho_6$, we have
\begin{align*}
\left|\frac{zf_2'(z)}{f_2(z)}\right| &= \left|\frac{\rho^4-(2+8b)\rho^3+(2-8b)\rho^2-(2+8b)\rho+1}{(1-\rho^2)(1-2(1+2b)\rho+\rho^2)}\right|\\
& = 1+\sin1\,\, \text{which lies in the boundary of}\,\, \Omega_S
\end{align*}
where $\Omega_S:= q_0(\D)$ is the image of  $\D$ under the mapping $q_0(z)= 1+\sin z$.

\item By \cite[Lemma 2.1, p. 3]{gandhi1}, the function $f \in \mathcal{G}_b^1$ belongs to the class $S_{\leftmoon}^*$ if the following inequality holds
\[\frac{2((1+b_1)r^3+2(1+b_1)r^2+(1+b_1)r)}{(1-r^2)(r^2+2b_1r+1)}\le 1-\sqrt{2}+\frac{1+r^2}{1-r^2}\]
which becomes
\begin{equation*}
  -\sqrt{2}r^4 +2(1+(1-\sqrt{2})b_1)r^3 +2(1+2b_1)r^2+2(1+(\sqrt{2}-1)b_1)r+(\sqrt{2}-2)\le 0.
\end{equation*}
Thus, the $S_{\leftmoon}^*$- radius  is the smallest root $\rho_7 \in (0,1)$ of (\ref{eqm4}).
At $z=\rho=\rho_7$, we have
\begin{align*}
\left|\left(\frac{zf_1'(z)}{f_1(z)}\right)^2-1\right|  &= \left| \left(\frac{\rho^4-2\rho^3+(2+8b)\rho^2-2\rho+1}{(1-\rho^2)(1-2(1+2b)\rho+\rho^2)}\right)^2-1\right|\\
%&=|(1-\sqrt{2})^2-1|\\
&=2 \left|\frac{zf_1'(z)}{f_1(z)}\right|.
\end{align*}

\item By \cite[Lemma 2.2 p. 202]{kumar}, the function $f \in \mathcal{G}_b^1$ belongs to the class $ S_R^*$ if
\[\frac{2((1+b_1)r^3+2(1+b_1)r^2+(1+b_1)r)}{(1-r^2)(r^2+2b_1r+1)}\le \frac{1+r^2}{1-r^2}+ 2-2 \sqrt{2}\]
equivalently
\begin{align*}
(1-2\sqrt{2})r^4+&2(1+2(1-\sqrt{2})b_1)r^3+\\
&2(1+2 b_1)r^2+2(1+2(\sqrt{2}-1)b_1)r+2\sqrt{2}-3 \le 0.
\end{align*}
Thus, the $S_R^*$- radius  is the smallest  root $\rho_8 \in (0,1)$ of (\ref{eqr4}).
As for $z=\rho=\rho_8$, we have
  \begin{align*}
  \left|\frac{zf_1'(z)}{f_1(z)}\right|  = \left|\frac{\rho^4-2\rho^3+(2+8b)\rho^2-2\rho+1}{(1-\rho^2)(1-2(1+2b)\rho+\rho^2)}\right|=2(\sqrt{2}-1) =\psi(-1)
 \end{align*}
 where $\psi(\D) = 1+ \dfrac{z}{k}\left(\dfrac{k+z}{k-z}\right) = 1+\dfrac{1}{k}z+ \dfrac{2}{k^2}z^2+ \cdots$, $k= \sqrt{2}+1$.

\item By \cite[Lemma 3.2, p. 10]{mnr1}, the function $f \in \mathcal{G}_b^1$ belongs to the class $ S_{RL}^*$ if
\begin{align*}
&\frac{2((1+b_1)r^3+2(1+b_1)r^2+(1+b_1)r)}{(1-r^2)(r^2+2b_1r+1)}\\
&\le ((1-(\sqrt{2}-\frac{1}{1-r^2})^2)^{1/2}-(1-(\sqrt{2}-\frac{1}{1-r^2})^2))^{1/2}
\end{align*}
which yields
\begin{align*}
&(2(1+b_1)r^3+4(1+b_1)r^2+2(1+b_1)r )^2
-(r^2+2b_1r+1)^2((1-r^2)\notag\\& \sqrt{-r^4-2  (\sqrt{2}-1 ) r^2+2  (\sqrt{2}-1 )}
+r^4+2 \left(\sqrt{2}-1\right) r^2-2 \sqrt{2}+2 \le0.
\end{align*}
Thus, the $S_{RL}^*$- radius  is the smallest  root $\rho_9 \in (0,1)$ of (\ref{eqrl4}).

\item By \cite[Lemma 3.1, p. 307]{grs},  the function $f \in \mathcal{G}_b^1$ belongs to the class $ S_{\gamma}^*$ if
\[\frac{2((1+b_1)r^3+2(1+b_1)r^2+(1+b_1)r)}{(1-r^2)(r^2+2b_1r+1)}\le \frac{1+r^2}{1-r^2}\sin(\frac{\pi\gamma}{2})\]
which  gives
\begin{align*}
&-\sin(\pi \gamma/2) r^4 +2(1+(1-\sin(\pi \gamma/2))b_1)r^3 \\ &+2(2(1+b_1)-\sin(\pi \gamma/2))r^2 +2(1+(1-\sin(\pi \gamma/2))b_1)r-\sin(\pi \gamma/2)\le 0.
\end{align*}

\item Using \cite[Lemma 2.2, p. 86]{wani}, the function $f \in \mathcal{G}_b^1$ belongs to the class $ S^*_{N_e}$ if
\[\frac{2((1+b_1)r^3+2(1+b_1)r^2+(1+b_1)r)}{(1-r^2)(r^2+2b_1r+1)}
\le \frac{5}{3}-\frac{1}{1-r^2}\]
equivalently
\begin{equation*}
5r^4+(6+16 b_1)r^3+(15+12 b_1)r^2+2(3+b_1)r-2 \le 0.
\end{equation*}
Thus, the $S^*_{N_e}$- radius  is the smallest  root $\rho_{11} \in (0,1)$ of (\ref{eqn4}).
Since $f_2$ is defined in (\ref{ext5}), then  at $z=-\rho=-\rho_{11}$, we have
\begin{align*}
  \left|\frac{zf_2'(z)}{f_2(z)}\right| &= \left|\frac{\rho^4-(2+8b)\rho^3+(2-8b)\rho^2-(2+8b)\rho+1}{(1-\rho^2)(1-2(1+2b)\rho+\rho^2)}\right| = \frac{5}{3},
  \end{align*}
  which shows that the estimate is best possible.

\item Using \cite[Lemma 2.2, p. 961]{goel}, the function $f \in \mathcal{G}_b^1$ belongs to the class $S^*_{SG}$ if the following inequality holds
\[\frac{2((1+b_1)r^3+2(1+b_1)r^2+(1+b_1)r)}{(1-r^2)(r^2+2b_1r+1)}
\le \frac{2e}{1+e}-\frac{1}{1-r^2}\]
equivalently if the following inequality holds
\begin{align*}
1-e+r(4 b_1+2 e+2)&+r^2 \left(4 b_1 e+4 b_1+6 e+6\right)+\\
&r^3 \left(8 b_1 e+4 b_1+2 e+2\right)+(3 e+1) r^4\le 0.
\end{align*}
Thus, the $S^*_{SG} $- radius is the smallest root $\rho_{12} \in (0,1)$ of (\ref{eqsg4}).
 For the function $f_2$ defined in (\ref{ext5}), let $w=zf_2'(z)/f_2(z)$. At $z=-\rho=-\rho_{12}$, we have following steps for sharpness
  \begin{align*}
  \left|\log\frac{w}{2-w}\right|
  %&= \left|\log \left(\frac{-z^4-(2+8b)z^3-(2-8b)z^2-(2+8b)z-1}{3 z^4+2(3+8b)z^3+(2-8b)z^2-2z-1}\right)\right|\\
   &= \left|\log \left(\frac{-\rho^4+(2+8b)\rho^3-(2-8b)\rho^2+(2+8b)\rho-1}{3 \rho^4-2(3+8b)\rho^3+(2-8b)\rho^2+2 \rho-1}\right)\right| = 1.
  \end{align*}
\qedhere
\end{enumerate}
\end{proof}
\begin{remark}
On taking $b=-1$,  parts(1)-(10) in Theorem \ref{ThmF4} scale down to \cite[Theorem 2.8]{AV} .
\end{remark}

\section{Radius estimates for the class $\mathcal{G}_b^2$}
Let the function $f$ be given by (\ref{eqf}) such that $\RE ((1+(1/z)) f(z)) >0$. Then $(1+(1/z)) f(z) = 1+(1+a_2)z+ \cdots \in \mathcal{P}$ and hence $-3\leq a_2 \leq 1$. We consider the class of function $f$ with fixed second coefficient whose  expansion is given by
$f(z) = z + 3b z^2 + \cdots$ ($|b| \le 1)$ .
\begin{definition}
For $|b| \le 1$, let $\mathcal{G}_b^2$ be a class of analytic functions given as
\[\mathcal{G}_b^2 := \left\{f \in \mathcal{A}: f(z) = z + 3b z^2 + \cdots, \RE \left(\left(1+\frac{1}{z}\right)f(z) \right)>0, z \in \D\right\} .\]
\end{definition}
Consider the function $f_3$ on $\D$ which is defined by
\begin{equation} \label{ext3}
f_3(z) = \frac{z(1+(1+3b)z+z^2)}{(1+z)(1-z^2)}.
\end{equation}
so that
\[ \frac{1+z}{z}f_3(z) = \frac{1+w_3(z)}{1-w_3(z)}\]
where \[w_3(z) = \dfrac{z(z+(1+3b)/2)}{(1+(1+3b)z/2)}\,\,\mbox{ with}\,\,|1+3b| \le 2\] is an analytic function fulfilling all the cases of Schwarz lemma and $\RE ((1+z)f_1(z)/z) >0$ in $\D$. Hence,  $f_3 \in \mathcal{G}_b^2$ that shows  the class  $\mathcal{G}_b^2$ is non-empty.

\begin{theorem}\label{ThmF3}
Let $b'=|1+3b|\leq2$. The following radius estimates for the class $\mathcal{G}_b^2 $ are computed:
\begin{enumerate}
\item $R_{S_c^*}(\mathcal{G}_b^2)=\rho_1$, where $\rho_1$ is the smallest root in $(0,1)$ of the equation
\begin{equation} \label{eqc3}
-r^4 +(3+2b')r^3+3(3+b')r^2+(3+b')r-2=0.
\end{equation}

\item $R_{S_{\sin}^*}(\mathcal{G}_b^2)=\rho_2$, where $\rho_2$ is the smallest root in $(0,1)$ of the equation
\begin{equation}\label{eqs3}
(1+\sin 1)r^4 +(1+(2+\sin 1)b')r^3+(5+b')r^2 +(1+(1-\sin 1)b')r-\sin 1=0.
\end{equation}
The result is sharp.

\item $R_{S_{\leftmoon}^*}(\mathcal{G}_b^2)=\rho_3$, where $\rho_3$ is the smallest root in $(0,1)$ of the equation
\begin{equation}\label{eqm3}
 (1-\sqrt{2})r^4 +(1+\sqrt{2}(\sqrt{2}-1)b')r^3\notag  +(3+b')r^2+(1+(\sqrt{2}-1)b')r+\sqrt{2}(1-\sqrt{2})=0.
\end{equation}

\item $R_{S_R^*}(\mathcal{G}_b^2)=\rho_4$, where $\rho_4$ is the smallest root in $(0,1)$ of the equation
\begin{equation}\label{eqr3}
2(1-\sqrt{2})r^4+(1+(3-2\sqrt{2})b')r^3+(3+b')r^2+(1+2(\sqrt{2}-1)b')r+2\sqrt{2}-3=0.
\end{equation}

\item $R_{S_{RL}^*}(\mathcal{G}_b^2)=\rho_5$, where $\rho_5$ is the smallest root in $(0,1)$ of the equation
\begin{align}\label{eqrl3}
&((1+b')r^3+(4+b')r^2+(1+b')r )^2
-(r^2+b'r+1)^2((1-r^2)\notag\\
&\sqrt{-r^4-2  (\sqrt{2}-1 ) r^2+2  (\sqrt{2}-1 )}
+r^4+2 \left(\sqrt{2}-1\right) r^2-2 \sqrt{2}+2=0.
\end{align}
\item $R_{S_{\gamma}^*}(\mathcal{G}_b^2)=\rho_6$, where $\rho_6$ is the smallest root in $(0,1)$ of the equation
\begin{equation*}%\label{eqg3}
 (1+b') r ^3  +(4-\sin(\pi \gamma/2)+b') r ^2 +(1+(1-\sin(\pi \gamma/2))b')r -\sin(\pi \gamma/2)\le 0.
\end{equation*}
\item
$R_{S_{N_e}^*}(\mathcal{G}_b^2)=\rho_7$, where $\rho_7$ is the smallest root in $(0,1)$ of the equation
\begin{equation}\label{eqn3}
 5r^4+(3+8b')r^3+3(5+b')r^2+(3+b')r-2=0.
\end{equation}
The result is sharp.

\item $R_{S_{SG}^*}(\mathcal{G}_b^2)=\rho_8$, where $\rho_8$ is the smallest root in $(0,1)$ of the equation
\begin{equation}\label{eqsg3}
 2er^4+(1+e+(1+3e)b')r^3+(1+e)(5+b')r^2+(1+e+2b')r+1-e=0.
\end{equation}
The result is sharp.
\end{enumerate}
\end{theorem}

\begin{proof}
Let the function $f \in \mathcal{F}_b^2$ be such that
\begin{equation} \label{eqfg2}
 \RE \left(\frac{1+z}{z}f(z) \right)>0\,\, (z \in D).
\end{equation}
Consider the function $h$ on $ \D $ given by
\begin{equation} \label{eqh(z)2}
h(z) = \left(1+\frac{1}{z}\right)f(z) =  1+(1+3b)z+ \cdots.
\end{equation}
In view of  (\ref{eqfg2}) and (\ref{eqh(z)2}), it is noted that  $h \in \mathcal{P}_{(1+3b)/2}$ and $f(z) = zh(z)/(1+z)$. A simple calculation gives
\begin{equation}\label{eqf(z)2}
\frac{zf'(z)}{f(z)} = \frac{zh'(z)}{h(z)}+\frac{1}{1+z}.
\end{equation}
Using Lemma \ref{lem1}, we get following inequality
\begin{equation} \label{eqt3}
\left|\frac{zh'(z)}{h(z)}\right| \le \frac{r}{1-r^2} \left(\frac{b'r^2+4r+b'}{r^2+b'r+1}\right).
\end{equation}
The bilinear transformation $\dfrac{1}{1+z}$ maps  $|z| \le r$ onto
\begin{equation} \label{eqtr}
\left|\frac{1}{1+z}-\frac{1}{1-r^2}\right|\le \frac{r}{1-r^2}.
\end{equation}

Using (\ref{eqtr}) and (\ref{eqt3}) along with (\ref{eqf(z)2}), we obtain

\begin{align*}%\label{eqdisk4}
\left|\frac{zf'(z)}{f(z)}- \frac{1}{1-r^2}\right| & \le \frac{r}{1-r^2}\left(\frac{b'r^2+4r+b'}{r^2+b'r+1}+1\right)\notag \\
&= \frac{(1+b')r^3+(4+b')r^2+(1+b')r}{(1-r^2)(r^2+b'r+1)}.
\end{align*}

It follows by the above inequality that
\[\RE\left(\frac{zf'(z)}{f(z)}\right) \ge\frac{-(1+b')r^3-(3+b')r^2-r+1}{-r^4-b'r^3+b'r+1}  \ge 0\]
whenever,
\[(1+b')r^3+(3+b')r^2+r-1 \le 0.\]

\begin{enumerate}
\item By making use of \cite[Lemma 2.5, p.926]{sharma},
%For $1/3 <a \le 5/3$, by Lemma \ref{lemc}
the function $f \in \mathcal{G}_b^2$ to be in the class $ S_c^*$ if
\[\frac{(1+b')r^3+(4+b')r^2+(1+b')r}{(1-r^2)(r^2+b'r+1)}\le \frac{1}{1-r^2}-\frac{1}{3}\]
or equivalently, if
\begin{equation*}
-r^4 +(3+2b')r^3+3(3+b')r^2+(3+b')r-2 \le 0.
\end{equation*}
Thus, the $S_c^*$-radius   is the smallest  root $\rho_1 \in (0,1)$ of (\ref{eqc3}).

\item Using \cite[Lemma 3.3, p.219]{cho},
% For $|a-1|\le \sin 1$, by Lemma \ref{lems}
 the function $f \in \mathcal{G}_b^2$  to be in the class $ S_{\sin}^*$ if
\[\frac{(1+b')r^3+(4+b')r^2+(1+b')r}{(1-r^2)(r^2+b'r+1)}\le \sin 1- \frac{r^2}{1-r^2}\]
equivalently
\begin{equation*}
(\sin 1+1)r^4 +(1+(2+\sin 1)b')r^3+(5+b')r^2 +(1+(1-\sin 1)b')r-\sin 1 \le0.
\end{equation*}
Thus, the $S_{\sin}^*$- radius for the class $\mathcal{G}_b^2$ is the smallest positive root $\rho_2 \in (0,1)$ of (\ref{eqs3}).
Consider the function $f_3$ defined in (\ref{ext3}). At $z=-\rho=-\rho_2$, we have
\begin{align*}
\left|\frac{zf_3'(z)}{f_3(z)}\right| &= \left|\frac{-3b\rho^3+(4-3b)\rho^2-(6b+1)\rho+1}{(1-\rho^2)(1-(1+3b)\rho+\rho^2)}\right|\\
& = 1+\sin1.
\end{align*}

\item By \cite[Lemma 2.1, p. 3]{gandhi1},
%By Lemma \ref{lemm}
the function $f \in \mathcal{G}_b^2$  to be in  the class $  S_{\leftmoon}^*$ if
\[\frac{(1+b')r^3+(4+b')r^2+(1+b')r}{(1-r^2)(r^2+b'r+1)}\le 1-\sqrt{2}+\frac{1}{1-r^2}\]
which implies
\begin{equation*}
 (1-\sqrt{2})r^4 +(1+\sqrt{2}(\sqrt{2}-1)b')r^3\notag  +(3+b')r^2+(1+(\sqrt{2}-1)b')r+\sqrt{2}(1-\sqrt{2}) \le 0.
\end{equation*}
Thus, the $S_{\leftmoon}^* $- radius   is the smallest   root $\rho_3 \in (0,1)$ of (\ref{eqm3}).

\item By \cite[Lemma 2.2 p. 202]{kumar},
%For $2(\sqrt{2}-1) < a \le \sqrt{2}$ by Lemma \ref{lemr}
the function $f \in \mathcal{G}_b^2$   to be in the class $  S_R^*$ if
\[\frac{(1+b')r^3+(4+b')r^2+(1+b')r}{(1-r^2)(r^2+b'r+1)}\le \frac{1}{1-r^2}+ 2-2 \sqrt{2}\]
equivalently
\begin{equation*}
2(1-\sqrt{2})r^4+(1+(3-2\sqrt{2})b')r^3+(3+b')r^2+(1+2(\sqrt{2}-1)b')r+2\sqrt{2}-3 \le 0.
\end{equation*}
Thus, the $S_R^* $ -radius   is the smallest   root $\rho_4 \in (0,1)$ of (\ref{eqr3}).

\item By \cite[Lemma 3.2, p. 10]{mnr1},
%For $\sqrt{2}/3 \le a < \sqrt{2}$, by Lemma \ref{lemrl}
the function $f \in \mathcal{G}_b^2$   to be in the class $  S_{RL}^*$ if
\begin{align*}
&\frac{(1+b')r^3+(4+b')r^2+(1+b')r}{(1-r^2)(r^2+b'r+1)}\\
&\quad\le ((1-(\sqrt{2}-\frac{1}{1-r^2})^2)^{1/2}-(1-(\sqrt{2}-\frac{1}{1-r^2})^2))^{1/2}
\end{align*}
which  gives
\begin{align*}
&((1+b')r^3+(4+b')r^2+(1+b')r )^2
-(r^2+b'r+1)^2((1-r^2)\notag\\
&\sqrt{-r^4-2  (\sqrt{2}-1 ) r^2+2  (\sqrt{2}-1 )}
+r^4+2 \left(\sqrt{2}-1\right) r^2-2 \sqrt{2}+2=0.
\end{align*}
Therefore, the $S_{RL}^*$-radius   is the smallest  root $\rho_5 \in (0,1)$ of (\ref{eqrl3}).

\item By \cite[Lemma 3.1, p. 307]{grs},
%By Lemma \ref{lemg}
the function $f \in \mathcal{G}_b^2$ to be in the class $ S_{\gamma}^*$ if
\[\frac{(1+b')r^3+(4+b')r^2+(1+b')r}{(1-r^2)(r^2+b'r+1)}\le \frac{1}{1-r^2}\sin(\frac{\pi \gamma}{2})\]
which simplifies to
\begin{equation*}
 (1+b')r^3  +(4-\sin(\pi \gamma/2)+b')r^2 +(1+(1-\sin(\pi \gamma/2))b')r-\sin(\pi \gamma/2)\le 0.
\end{equation*}

\item Using \cite[Lemma 2.2, p. 86]{wani},
%For $1 \le a <5/3$, by Lemma \ref{lemne}
the function $f \in \mathcal{G}_b^2$  to be in the class $ S^*_{N_e}$ if
\[\frac{(1+b')r^3+(4+b')r^2+(1+b')r}{(1-r^2)(r^2+b'r+1)}\le \frac{5}{3}-\frac{1}{1-r^2}\]
which is equivalent to
\begin{equation*}
5r^4+(3+8b')r^3+3(5+b')r^2+(3+b')r-2 \le 0.
\end{equation*}
Thus, the $S^*_{N_e}$- radius
 is the smallest  root $\rho_7 \in (0,1)$ of (\ref{eqn3}).
Consider the function $f_3$ defined in (\ref{ext3}). At $z=-\rho=-\rho_7$, we have
\begin{align*}
\left|\frac{zf_3'(z)}{f_3(z)}\right| &= \left|\frac{-3b\rho^3+(4-3b)\rho^2-(6b+1)\rho+1}{(1-\rho^2)(1-(1+3b)\rho+\rho^2)}\right|\\
& = \frac{5}{3}  \,\, \text{which lies in the boundary of}\,\, \partial\Omega_{N_e}
\end{align*}
where $\Omega_{Ne}$ is  the nephroid shaped region.
This shows that the result is sharp.

\item Using \cite[Lemma 2.2, p. 961]{goel},
%For $1 \le a <2e/(1+e)$, by Lemma \ref{lemsg}
the function $f \in \mathcal{G}_b^2$ belongs to the class $ S^*_{SG}$ if
\[\frac{(1+b')r^3+(4+b')r^2+(1+b')r}{(1-r^2)(r^2+b'r+1)}\le \frac{2e}{1+e}-\frac{1}{1-r^2}\]
equivalently if the following inequality holds
\begin{equation*}
2er^4+(1+e+(1+3e)b')r^3+(1+e)(5+b')r^2+(1+e+2b')r+1-e\le 0.
\end{equation*}
Thus, the $S^*_{SG} $- radius for the class $\mathcal{G}_b^2$ is a smallest root $\rho_8 \in (0,1)$ of (\ref{eqsg3}).
Consider the function $f_3$ defined in (\ref{ext3}). Let $w= zf_3'(z)/f_3(z)$. At $z= -\rho=-\rho_8$,  we have
\begin{align*}
\left|\log\left(\frac{w}{2-w}\right)\right|
%&= \left|\log\left(\frac{-3bz^3-(4-3b)z^2-(6b+1)z-1}{2 z^4+(2+9b)z^3+(4-3b)z^2-z-1}\right)\right|\\
&=\left|\log\left(\frac{3b\rho^3-(4-3b)\rho^2+(6b+1)\rho-1}{2 \rho^4-(2+9b)\rho^3+(4-3b) \rho^2+\rho-1}\right)\right|=1
\end{align*}
which shows the sharpness.\qedhere
\end{enumerate}
\end{proof}

\begin{remark}
For $b=-1$  in parts (1)-(6) of Theorem \ref{ThmF3}, we obtain \cite[Theorem 2.6]{AV} .
\end{remark}


\begin{thebibliography}{99}

\bibitem{agh18} Aghalary, R., Mohammadian, A., \emph{Radii of harmonic mapping with fixed second coefficients in the plane}, Stud. Univ. Babe\c{s}-Bolyai Math., \textbf{63}(2018), no. 2, 189--201.

\bibitem{ahuja87} Ahuja, O.P., Silverman, H., \emph{Extreme points of families of univalent functions with fixed second coefficient}, Colloq. Math.,  \textbf{54}(1987), no. 1, 127--137.
\bibitem{ali1} Ali, R.M.,  Cho, N.E., Jain, N.K., Ravichandran, V., \emph{Radii of starlikeness and convexity for functions with fixed second coefficient defined by subordination}, Filomat,  \textbf{26} (2012), no. (3), 553--561.

\bibitem{ali}  Ali, R.M.,  Jain, N.K., Ravichandran, V.,  \emph{Radii of starlikeness associated with the lemniscate of Bernoulli and the left-half plane}, Appl. Math. Comput.,  \textbf{218}(2012), no. 11, 6557--6565.
\bibitem{juneja} Ali, R.M., Jain, N.K., Ravichandran, V.,  \emph{On the radius constants for classes of analytic functions}, Bull. Malays. Math. Sci. Soc. (2), {\textbf{36}}(2013), no. 1, 23--38.
 \bibitem{ANV} Ali, R.M., Nagpal, S., Ravichandran, V.,  \emph{Second-order differential subordination for analytic functions with fixed initial coefficient}, Bull. Malays. Math. Sci. Soc. (2),  {\textbf{34}}(2011), no. 3, 611--629.

\bibitem{ali2} Ali, R.M., Kumar, V.,  Ravichandran, V., Kumar, S.S.,  \emph{Radius of starlikeness for analytic functions with fixed second coefficient}, Kyungpook Math. J.,  {\textbf{57}}(2017), no. 3, 473--492.
\bibitem{ravi} Anand, S., Jain, N.K., Kumar, S., \emph{Sharp Bohr radius constants for certain analytic functions}, Bull. Malays. Math. Sci. Soc.,  {\textbf{44}}(2021), no. 3, 1771--1785.

\bibitem{cho19} Cho, N. E., Kumar, S., Kumar, V., Ravichandran, V., \emph{ Differential subordination and radius estimates for starlike functions associated with the Booth lemniscate}, Turkish J. Math. {\bf 42} (2018), no.~3, 1380--1399.

\bibitem{cho} Cho, N.E., V. Kumar, V.,   Kumar, S.S.,   Ravichandran, V., \emph{Radius problems for starlike functions associated with the sine function}, Bull. Iranian Math. Soc.,  {\textbf{45}}(2019), no. 1, 213--232.


\bibitem{gandhi1}Gandhi S.,  Ravichandran, V., \emph{Starlike functions associated with a lune}, Asian-Eur. J. Math.,  {\textbf{10}}(2017), no. 4, 12pp.

\bibitem{grs}Gangadharan, A., Ravichandran, V.,  Shanmugam, T.N.,  \emph{Radii of convexity and strong starlikeness for some classes of analytic functions}, J. Math. Anal. Appl.,  {\textbf{211}}(1997), no. 1, 301--313.

 \bibitem{goel}Goel, P.,  Kumar, S.S.,  \emph{Certain class of starlike functions associated with modified sigmoid function}, Bull. Malays. Math. Sci. Soc.,  {\textbf{43}}(2020), no. 1, 957--991.

\bibitem{Gronwall} Gronwall, T.H.,  \emph{On the distortion in conformal mapping when the second coefficient
in the mapping function has an assigned value}, Nat. Acad. Proc., \textbf{6}(1920), 300-–302.

%\bibitem{gupta80}Gupta, V.P., Jain, P.K., Ahmad, I.: On the radius of univalence of certain classes of analytic functions with fixed second coefficient, Rend. Mat. (6) {\bf 12}(3-4), 423--430  (1979)(1980)


\bibitem{kanaga} Kanaga, R.,  Ravichandran, V.,  \emph{Starlikeness for certain close-to-star functions}, Hacet. J. Math. Stat., {\textbf{50}}(2021), no. 2, 414-–432.

\bibitem{kumar21} Kumar, S., Rai, P.,  \c{C}etinkaya, A., \emph{ Radius estimates of certain analytic functions}, Honam Math. J. {\bf 43} (2021), no.~4, 627--639.

\bibitem{kumar} Kumar, S., Ravichandran, V.,  \emph{A subclass of starlike functions associated with a rational function}, Southeast Asian Bull. Math.,  {\textbf{40}}(2016), no. 2, 199--212.

% \bibitem{kumar83}Kumar, V., On univalent functions with fixed second coefficient, Indian J. Pure Appl. Math. {\bf 14}(11), 1424--1430 (1983)

\bibitem{sushil17} Kumar, S., Ravichandran, V., Verma, S.,  \emph{Initial coefficients of starlike functions with real coefficient}, Bull. Iranian Math. Soc., {\textbf{43}}(2017), no. 6, 1837--1854.

\bibitem{lee}Lee, S.K.,  Ravichandran, V., Supramaniam, S.,  \emph{Applications of differential subordination for functions with fixed second coefficient to geometric function theory}, Tamsui Oxf. J. Inf. Math. Sci.,  {\textbf{29}}(2013), no. 2, 267--284.

\bibitem{lee1}Lee, S.K.,   Khatter, K., Ravichandran, V.,  \emph{Radius of Starlikeness for Classes of Analytic Functions}, Bull. Malays. Math. Sci. Soc.,  {\textbf{43}}(2020), no. 6, 4469--4493.

\bibitem{MM} Ma, W.C., Minda, D.,  \emph{A unified treatment of some special classes of univalent functions}, in {Proceedings of the Conference on Complex Analysis (Tianjin, 1992)}, 157--169, Conf. Proc. Lecture Notes Anal., I, Int. Press, Cambridge (1992).


%\bibitem{gregor} MacGregor, T.H.: The radius of convexity for starlike functions of order ${1\over 2}$, Proc. Amer. Math. Soc. {\bf 14}, 71--76 (1963)
%\bibitem{gregor1}MacGregor, T.H.,: The radius of univalence of certain analytic functions, Proc. Amer. Math. Soc. {\bf 14}, 514--520 (1963)
%\bibitem{gregor2}T. H. MacGregor, A class of univalent functions, Proc. Amer. Math. Soc. {\bf 15} (1964), 311--317.

\bibitem{carty}McCarty, C.P.,  \emph{Functions with real part greater than $\alpha $}, Proc. Amer. Math. Soc.,  {\textbf{35}}(1972), 211--216.

\bibitem{mnr} Mendiratta, R.,  Nagpal, S.,  Ravichandran, V.,  \emph{On a subclass of strongly starlike functions associated with exponential function}, Bull. Malays. Math. Sci. Soc.,  {\textbf{38}}(2015), no. 1, 365--386.
\bibitem{mnr1} Mendiratta, R., Nagpal, S., Ravichandran, V.,  \emph{A subclass of starlike functions associated with left-half of the lemniscate of Bernoulli}, Internat. J. Math., {\textbf{25}}(2014), no. 9, 1450090, 17 pp.
%\bibitem{mendiratta1} Mendiratta, R.,  Ravichandran, V., Livingston problem for close-to-convex functions
%with fixed second coefficient, Jnanabha,\textbf{ 43}, 107–-122 (2013).

\bibitem{mnr2} Mendiratta, R., Nagpal, S., Ravichandran, V., \emph{Second-order differential superordination for analytic functions with fixed initial coefficient}, Southeast Asian Bull. Math., {\textbf{39}}(2015), no. 6, 851--864.

%\bibitem{mnr3} Mendiratta, R.,  Nagpal, S.,  Ravichandran, V.: Radii of starlikeness and convexity for analytic functions with fixed second coefficient satisfying certain coefficient inequalities, Kyungpook Math. J. {\bf 55}(2), 395--410 (2015)

%\bibitem{nagpal}S. Nagpal\ and\ V. Ravichandran, Applications of the theory of differential subordination for functions with fixed initial coefficient to univalent functions, Ann. Polon. Math. {\bf 105} (2012), no.~3, 225--238.



\bibitem{nehari}Nehari, Z., {\emph{Conformal mapping}}, McGraw-Hill Book Co., Inc., New York, Toronto, London, 1952.


\bibitem{padman} Padmanabhan, K.S.,  Parvatham, R.,  \emph{Some applications of differential subordination}, Bull. Austral. Math. Soc.,  {\textbf{32}}(1985), no. 3, 321--330.

\bibitem{pas89} Pa\v{s}kuleva, D.Z.,  \emph{On the radius of alpha-starlikeness for starlike functions of order beta and fixed second coefficient}, Pliska Stud. Math. Bulgar.,  {\textbf{10}}(1989), 157--163.

\bibitem{raina} Raina, R.K.,  Sok\'{o}\l, J.,  \emph{Some properties related to a certain class of starlike functions}, C. R. Math. Acad. Sci. Paris,  {\textbf{\bf }353}(2015), no. 11, 973--978.
%\bibitem{ratti}J. S. Ratti, The radius of convexity of certain analytic functions, Indian J. Pure Appl. Math. {\bf 1} (1970), no.~1, 30--36.
%\bibitem{ratti1}J. S. Ratti, The radius of univalence of certain analytic functions, Math. Z. {\bf 107} (1968), 241--248.

%\bibitem{rob} Robertson, M.S.: Certain classes of starlike functions, Michigan Math. J. {\bf 32}(2), 135--140 (1985)

\bibitem{Ronn93} R\o nning, F., \emph{Uniformly convex functions and a corresponding class of starlike functions}, Proc. Amer. Math. Soc.,  {\textbf{\bf }118}(1993), no. 1, 189--196.

\bibitem{AV}  Sebastian, A., Ravichandran, V.,  \emph{Radius of starlikeness of certain analytic functions}, Math. Slovaca,  {\textbf{\bf }71}(2021), no. 1, 83--104.

\bibitem{TN1} Shanmugam, T.N.,  \emph{Convolution and differential subordination}, Internat. J. Math. Math. Sci.,  {\textbf{12}}(1989), no. 2, 333--340.


\bibitem{TN} Shanmugam, T.N., Ravichandran, V.,  \emph{Certain properties of uniformly convex functions}, in { Computational methods and function theory 1994 (Penang)}, 319--324, Ser. Approx. Decompos., 5, World Sci. Publ., River Edge, NJ (1994).



\bibitem{sharma} Sharma, K., Jain, N.K., Ravichandran, V.,  \emph{Starlike functions associated with a cardioid}, Afr. Mat.,  {\textbf{27}}(2016), no. 5-6, 923--939.


\bibitem{SS1} Sok\'{o}\l, J., Stankiewicz, J.,  \emph{Radius of convexity of some subclasses of strongly starlike functions}, Zeszyty Nauk. Politech. Rzeszowskiej Mat., {\textbf{19}}(1996), 101--105.

%\bibitem{sokol} Sok\'{o}\l, J.: Radius problems in the class $S^*_L$, Appl. Math. Comput. {\bf 214}(2), 569--573(2009)


%\bibitem{Ural87} Uralegaddi, B.A., Patil, H.S.: The starlike functions of order $\alpha$ and type $\beta$ with fixed second coefficients, J. Karnatak Univ. Sci. {\bf 32}, 219--226 (1987)

\bibitem{wani} Wani, L.A.,  Swaminathan, A.,  \emph{Starlike and convex functions associated with a nephroid domain}, Bull. Malays. Math. Sci. Soc.,  {\textbf{44}}(2021),  no. 1, 79--104.

%\bibitem{nam1} Name1, I., \emph{A well known title}, Marcel Dekker, 1976.
%\bibitem{nam2} Name2, W., \emph{On an interesting formula}, Journal name, \textbf{7}(1990), no. 1, 243-255.

\end{thebibliography}
\end{document}